\documentclass[12pt]{amsart}      
\usepackage{amssymb}
\usepackage{eucal}
\usepackage{graphicx}
\usepackage{amsmath}
\usepackage{amscd}
\usepackage[all]{xy}           
\usepackage{longtable, lscape}
\usepackage{rotating}

\usepackage{amsfonts,latexsym}
\usepackage{xspace}
\usepackage{epsfig}
\usepackage{float}
\usepackage{color}
\usepackage{fancybox}
\usepackage{colordvi}
\usepackage{multicol}

\newtheorem{theorem}{Theorem}[section]
\newtheorem{prop}[theorem]{Proposition}
\newtheorem{defn}[theorem]{Definition}

\newtheorem{coro}[theorem]{Corollary}
\newtheorem{thm-def}[theorem]{Theorem-Definition}
\newtheorem{prop-def}[theorem]{Proposition-Definition}
\newtheorem{coro-def}[theorem]{Corollary-Definition}

\newtheorem{remark}[theorem]{Remark}

\newcommand{\nc}{\newcommand}
\nc{\tred}[1]{\textcolor{red}{#1}}
\nc{\tblue}[1]{\textcolor{blue}{#1}}
\nc{\tgreen}[1]{\textcolor{green}{#1}}
\nc{\tpurple}[1]{\textcolor{purple}{#1}}
\nc{\btred}[1]{\textcolor{red}{\bf #1}}
\nc{\btblue}[1]{\textcolor{blue}{\bf #1}}
\nc{\btgreen}[1]{\textcolor{green}{\bf #1}}
\nc{\btpurple}[1]{\textcolor{purple}{\bf #1}}

\renewcommand{\Bbb}{\mathbb}
\renewcommand{\frak}{\mathfrak}
\newcommand{\efootnote}[1]{}
\renewcommand{\textbf}[1]{}
\newcommand{\delete}[1]{}
\nc{\dfootnote}[1]{{}}          
\nc{\ffootnote}[1]{\dfootnote{#1}}
\nc{\mfootnote}[1]{\footnote{#1}} 
\nc{\ofootnote}[1]{\footnote{\tiny Older version: #1}} 

\delete{
\nc{\mlabel}[1]{\label{#1}  
{\hfill \hspace{1cm}{\bf{{\ }\hfill(#1)}}}}
\nc{\mcite}[1]{\cite{#1}{{\bf{{\ }(#1)}}}}  
\nc{\mref}[1]{\ref{#1}{{\bf{{\ }(#1)}}}}  
\nc{\mbibitem}[1]{\bibitem[\bf #1]{#1}} 
}

\nc{\mlabel}[1]{\label{#1}}  
\nc{\mcite}[1]{\cite{#1}}  
\nc{\mref}[1]{\ref{#1}}  
\nc{\mbibitem}[1]{\bibitem{#1}} 

\nc{\sbar}{\, {\scriptstyle{|\hspace{-.08cm}|\hspace{-.08cm}|
\hspace{-.08cm}|\hspace{-.08cm}|\hspace{-.08cm}|
\hspace{-.08cm}|\hspace{-.08cm}|\hspace{-.08cm}|\hspace{-.08cm}|
\hspace{-.08cm}|\hspace{-.08cm}|\hspace{-.08cm}|
\hspace{-.08cm}|\hspace{-.08cm}|\hspace{-.08cm}|\hspace{-.08cm}|
\hspace{-.08cm}|\hspace{-.08cm}|\hspace{-.08cm}|
\hspace{-.08cm}|\hspace{-.08cm}|\hspace{-.08cm}|\hspace{-.08cm}|
\hspace{-.08cm}|\hspace{-.08cm}|\hspace{-.08cm}|
\hspace{-.08cm}|\hspace{-.08cm}|}\, }}

\nc{\wvec}[2]{{\scriptsize{\big [ \!\!
    \begin{array}{c} #1 \\ #2 \end{array} \!\! \big ]}}}
\nc{\lp}{\big ( }
\nc{\rp}{\big ) }
\nc{\lb}{\!\left \langle }
\nc{\rb}{\right \rangle\! }

\nc{\bin}[2]{ (_{\stackrel{\scs{#1}}{\scs{#2}}})}  
\nc{\binc}[2]{ \big (\!\! \begin{array}{c} \scs{#1}\\
    \scs{#2} \end{array}\!\! \big )}  
\nc{\bincc}[2]{  \left ( {\scs{#1} \atop
    \vspace{-1cm}\scs{#2}} \right )}  
\nc{\bs}{\bar{S}}
\nc{\cosum}{\sqsubset}
\nc{\la}{\longrightarrow}
\nc{\rar}{\rightarrow}
\nc{\dar}{\downarrow}
\nc{\dap}[1]{\downarrow \rlap{$\scriptstyle{#1}$}}
\nc{\uap}[1]{\uparrow \rlap{$\scriptstyle{#1}$}}
\nc{\defeq}{\stackrel{\rm def}{=}}
\nc{\disp}[1]{\displaystyle{#1}}
\nc{\dotcup}{\ \displaystyle{\bigcup^\bullet}\ }
\nc{\gzeta}{\bar{\zeta}}
\nc{\hcm}{\ \hat{,}\ }
\nc{\hts}{\hat{\otimes}}
\nc{\barot}{{\otimes}}
\nc{\free}[1]{\bar{#1}}
\nc{\uni}[1]{\tilde{#1}}
\nc{\hcirc}{\hat{\circ}}
\nc{\lleft}{[}
\nc{\lright}{]}
\nc{\curlyl}{\left \{ \begin{array}{c} {} \\ {} \end{array}
    \right .  \!\!\!\!\!\!\!}
\nc{\curlyr}{ \!\!\!\!\!\!\!
    \left . \begin{array}{c} {} \\ {} \end{array}
    \right \} }
\nc{\longmid}{\left | \begin{array}{c} {} \\ {} \end{array}
    \right . \!\!\!\!\!\!\!}
\nc{\ora}[1]{\stackrel{#1}{\rar}}
\nc{\ola}[1]{\stackrel{#1}{\la}}
\nc{\ot}{\otimes}
\nc{\mot}{{{\sbar}}}
\nc{\otm}{\bar{\sbar}}
\nc{\scs}[1]{\scriptstyle{#1}}
\nc{\subv}{{^{\star}}}
\nc{\cov}{{^{\sharp}}}
\nc{\mrm}[1]{{\rm #1}}
\nc{\dirlim}{\displaystyle{\lim_{\longrightarrow}}\,}
\nc{\invlim}{\displaystyle{\lim_{\longleftarrow}}\,}
\nc{\proofbegin}{\noindent{\bf Proof: }}
\nc{\proofend}{$\blacksquare$ \vspace{0.3cm}}
\nc{\sha}{{\mbox{\cyr X}}}  
\nc{\shap}{{\mbox{\cyrs X}}} 
\nc{\shpr}{\diamond}    
\nc{\shplus}{\shpr^+}
\nc{\shprc}{\shpr_c}    
\nc{\msh}{\ast}
\nc{\vep}{\varepsilon}
\nc{\labs}{\mid\!}
\nc{\rabs}{\!\mid}
\nc{\FG}{\mrm{FG}}
\nc{\fp}{\mrm{fp}} \nc{\rchar}{\mrm{char}} \nc{\Fil}{\mrm{Fil}}
\nc{\gmzvs}{gMZV\xspace}
\nc{\gmzv}{gMZV\xspace}
\nc{\mzv}{MZV\xspace}
\nc{\mzvs}{MZVs\xspace}
\nc{\MZV}{\mrm{MZV}}
\nc{\Hom}{\mrm{Hom}} \nc{\id}{\mrm{id}} \nc{\im}{\mrm{im}}
\nc{\incl}{\mrm{incl}} \nc{\map}{\mrm{Map}} \nc{\mchar}{\rm char}
\nc{\nz}{\rm NZ} \nc{\supp}{\mathrm Supp}

\nc{\Alg}{\mathbf{Alg}}
\nc{\Bax}{\mathbf{Bax}}
\nc{\bff}{\mathbf f}
\nc{\bfk}{{\bf k}}
\nc{\bfone}{{\bf 1}}
\nc{\bfx}{\mathbf x}
\nc{\bfy}{\mathbf y}
\nc{\base}[1]{\bfone^{\otimes ({#1}+1)}} 
\nc{\Cat}{\mathbf{Cat}}

\nc{\detail}{\marginpar{\bf More detail}
    \noindent{\bf Need more detail!}
    \svp}
\nc{\Int}{\mathbf{Int}}
\nc{\Mon}{\mathbf{Mon}}
\nc{\remarks}{\noindent{\bf Remarks: }}
\nc{\Rings}{\mathbf{Rings}}
\nc{\Sets}{\mathbf{Sets}}

\nc{\BA}{{\Bbb A}} \nc{\CC}{{\Bbb C}} \nc{\DD}{{\Bbb D}}
\nc{\EE}{{\Bbb E}} \nc{\FF}{{\Bbb F}} \nc{\GG}{{\Bbb G}}
\nc{\HH}{{\Bbb H}} \nc{\LL}{{\Bbb L}} \nc{\NN}{{\Bbb N}}
\nc{\KK}{{\Bbb K}} \nc{\QQ}{{\Bbb Q}} \nc{\RR}{{\Bbb R}}
\nc{\TT}{{\Bbb T}} \nc{\VV}{{\Bbb V}} \nc{\ZZ}{{\Bbb Z}}

\nc{\cala}{{\mathcal A}} \nc{\calc}{{\mathcal C}}
\nc{\cald}{{\mathcal D}} \nc{\cale}{{\mathcal E}}
\nc{\calf}{{\mathcal F}} \nc{\calg}{{\mathcal G}}
\nc{\calh}{{\mathcal H}} \nc{\cali}{{\mathcal I}}
\nc{\call}{{\mathcal L}} \nc{\calm}{{\mathcal M}}
\nc{\caln}{{\mathcal N}} \nc{\calo}{{\mathcal O}}
\nc{\calp}{{\mathcal P}} \nc{\calr}{{\mathcal R}}
\nc{\cals}{{\mathcal S}}
\nc{\calt}{{\mathcal T}} \nc{\calw}{{\mathcal W}}
\nc{\calk}{{\mathcal K}} \nc{\calx}{{\mathcal X}}
\nc{\CA}{\mathcal{A}}

\nc{\fraka}{{\frak a}}
\nc{\frakA}{{\frak A}}
\nc{\frakb}{{\frak b}}
\nc{\frakB}{{\frak B}}
\nc{\frakH}{{\frak H}}
\nc{\frakM}{{\frak M}}
\nc{\bfrakM}{\overline{\frakM}}
\nc{\frakm}{{\frak m}}
\nc{\frakP}{{\frak P}}
\nc{\frakN}{{\mathfrak N}}
\nc{\frakp}{{\frak p}}
\nc{\frakS}{{\frak S}}

\font\cyr=wncyr10
\font\cyrs=wncyr7
\nc {\p}{\partial }

\nc{\redt}[1]{\textcolor{red}{#1}} \nc{\li}[1]{\textcolor{red}{Li:
#1}} \nc{\zb}[1]{\textcolor{blue}{Bin: #1}}
\nc{\zhb}[1]{\textcolor{red}{Bin: #1}}
\begin{document}

\title[Differential Birkhoff decomposition and MZVs]{Differential Birkhoff decomposition and
the renormalization of multiple zeta values}
%
\author{Li Guo}
\address{Department of Mathematics and Computer Science,
         Rutgers University,
         Newark, NJ 07102}
\email{liguo@newark.rutgers.edu}
\author{Bin Zhang}
\address{Yangtze Center of Mathematics,
Sichuan University, Chengdu, 610064, P. R. China }
\email{binzhang@mpim-bonn.mpg.de}

\maketitle

\begin{abstract}
In the Hopf algebra approach of Connes and Kreimer on renormalization of quantum field theory, the renormalization process is views as a special case of
the Algebraic Birkhoff Decomposition. We give a differential algebra variation of this decomposition and apply this to the study of multiple zeta values.
\end{abstract}



\setcounter{section}{0}

\section{Introduction}

This paper applies the renormalization method in quantum field theory to the study of multiple zeta values when the defining sums of the multiple zeta values are divergent, with an emphasis on the differential structure underlying the renormalization process.

As we will see later, such divergence of the infinite sums cannot be cured by traditional mathematical methods, such as analytic continuation as in the case of one variable (Riemann) zeta function. On the other hand, theoretical physicists have dealt with similar divergencies in quantum field theory for several decades with such a great success that the related renormalization process is regarded as one of the greatest achievements in modern physics. Nevertheless mathematicians have been skeptical about the soundness of the mathematical foundation of the renormalization process. This situation is changed by the recent seminal work of Connes and Kreimer~\mcite{C-K1,C-K2} which puts the renormalization process in a more mathematical framework of Algebraic Birkhoff Decomposition. Such a framework not only reveals the mathematical structure underlying the physics process, it also makes it possible for this physics process to be adapted to treat other apparently unrelated problems in mathematics, in particular the divergence of multiple zeta values.

This is the goal of this and the companying paper~\mcite{G-Z} which complete each other. This paper consists of three sections.
Section~\mref{sec:mzvs}
motivates our renormalization approach through examples and special cases.
In Section~\mref{sec:setup}, we obtain the differential variation of the Algebraic Birkhoff Decomposition of Connes and Kreimer~\mcite{C-K1,C-K2}. We also extend the well-known quasi-shuffle Hopf algebra of Hoffman~\mcite{Ho2} in two directions.
In one direction we consider such an algebra that is generated by a module instead of by a locally finite set.
In the other direction, we consider quasi-shuffle Hopf algebras in the context of differential Hopf algebras instead of Hopf algebras. This context allows us to apply the Differential Algebraic Birkhoff Decomposition in Section~\mref{sec:gmzv} to study differential properties of renormalized multiple zeta values.

This paper highlights the differential aspect of the renormalization process. As is well-known, for the convergent multiple zeta values, the interaction between the quasi-shuffle product from their summation representations and the shuffle product from their integration representations plays a major role in their study. In the renormalized approach of divergent multiple zeta values, the integration representation is not available. On the other hand, differentiation is  essential in the study of renormalized multiple zeta values. We interpret this in the Algebraic Birkhoff Decomposition and give an application on generating functions of renormalized multiple zeta values.

Parts of this paper, such as the examples in Section~\mref{sec:mzvs} and the general construction of quasi-shuffle Hopf algebras in Section~\mref{sec:setup}, were included in an early version of~\mcite{G-Z}, but was taken out of the final version of that paper to limit the size of the paper. Some readers have found the examples helpful in understanding the renormalization process in general.
Further the general construction of quasi-shuffle algebra in Section~\mref{sec:setup} has been referred to or related to in a number of papers, such as that of D. Manchon and S. Payche~\mcite{M-P2} on renormalization of multiple zeta values, of J. Zhao~\mcite{Zh} on renormalization of multiple $q$-zeta values and of F. Menous~\mcite{Me} on Birkhoff decomposition. In particular, for the last two references, the case treated in~\mcite{G-Z} is not adequate and the general construction is needed. Also considering its application in later part of this paper in the differential context, we have decided to make the details of the construction available.

\medskip

\noindent {\bf Acknowledgements.} Both authors thank the Max-Planck
Institute for Mathematics in Bonn for the stimulating environment.
The first named author thanks the NSF grant DMS 0505445 for
support. Thanks also go to Robert Sczech, especially for suggesting the term
algebraic continuation, and to Herbert Gangl, especially for
suggesting the bar notation.

\section{Renormalization of MZVs: motivation and examples}
\mlabel{sec:mzvs}

Multiple zeta values (MZVs) are defined to be the convergent sums
\begin{equation} \zeta(s_1,\cdots, s_k)=\sum_{n_1>\cdots>n_k>0}
    \frac{1}{n_1^{s_1}\cdots n_k^{s_k}}
\mlabel{eq:mzv}
\end{equation}
where $s_1,\cdots,s_k$ are positive integers with $s_1>1$. Since the papers of Hoffman~\mcite{Ho0} and Zagier~\mcite{Za} in the early 1990s, their study have attracted interests from several areas of mathematics and physics~\mcite{B-B-B-L,B-K,Ca1,Go3,G-M,Ho2,Krei}, including number theory, combinatorics, algebraic geometry and mathematical physics.

In order to study of the multiple variable function
$\zeta(s_1,\cdots,s_k)$ at integers $s_1,\cdots,s_k$ where the
defining sum (\mref{eq:mzv}) is divergent, one first tries to use
the analytic continuation, as in the case of the one variable case
of the Riemann zeta function. Such an analytic continuation is
achieved recently in~\mcite{AET,A-K,Ma1,Ma2,Zh2}. Unfortunately,
unlike in the one variable case, the multiple zeta function in
Eq.~(\mref{eq:mzv}) is still undefined at most non-positive integers even with the analytic continuation.
Nevertheless, possible definitions of multiple zeta functions at
certain non-positive integers were proposed in \mcite{AET,A-K} by
making use of the analytic continuation. Let us briefly recall
these previous progresses before introducing our approach by
renormalization.

\subsection{Earlier approach by analytic continuation}
Analytic continuation of $\zeta(s_1,s_2)$ has been considered as early as 1949 by Atkinson~\mcite{At2} with applications to the study of the asymptotic behavior of the ``mean values" of zeta-function.

Through the more recent work of Zhao~\mcite{Zh2} and Akiyama-Egami-Tanigawa~\mcite{AET},
we know that $\zeta(s_1,\cdots,s_k)$ can be meromorphically continued to $\CC^k$
with singularities on the subvarieties
\begin{eqnarray*}
s_1&=&1;\\
s_1+s_2&=&2,1,0,-2,-4, \cdots; {\rm\ and\ } \mlabel{eq:pole}\\
\sum_{i=1}^j
s_{i} &\in& \ZZ_{\leq j}\ (3\leq j\leq k).
\end{eqnarray*}
We also have some control of the behavior near the singularities.
For example, near $(0,0)$,
$$ \zeta(s_1,s_2)=\frac{5s_1+4s_2}{12(s_1+s_2)}+R_2(s_1,s_2)$$
where $R_2(s_1,s_2)$
is an entire function near $(0,0)$ with $R_2(0,0)=0$.

In~\mcite{AET,A-K}, several definitions were proposed for multiple zeta functions at $(s_1,\cdots,s_k)$ where $s_i$ are all non-positive. Some of the definitions of $\zeta(s_1,\cdots,s_k)$ are \begin{equation}
\lim_{r_1\to s_1}\cdots\lim_{r_k\to s_k} \zeta(r_1,\cdots,r_k),\quad
\lim_{r_k\to s_k}\cdots\lim_{r_1\to s_1} \zeta(r_1,\cdots,r_k),\quad
\lim_{r\to 0} \zeta(s_1+r,\cdots,s_k+r).
\mlabel{eq:jmzv}
\end{equation}
Naturally they give different values. In the case of
$\zeta(s_1,s_2)$ at $(s_1,s_2)=(0,0)$, the proposed values are
$5/12,1/3, 3/8$ respectively according to the above three
definitions. In fact, by letting $(r_1,r_2)$ approach to $(0,0)$
along different paths, one can get any value, as
well as the infinity, to be the limit. Even though some good
properties of the variously defined non-positive MZVs were
obtained in these studies, they fell short of the analogous
properties of the positive MZVs, especially the double shuffle
relations.

\subsection{An illustration of the renormalization approach}
Our approach to define MZVs where the sum (\mref{eq:mzv}) is divergent is adapted from a renormalization procedure (dimensional
regularization plus minimal subtraction) in quantum field theory
(QFT). This definition coincides with the
usual definition of MZVs when they are defined. Our
extended MZVs also satisfy the quasi-shuffle
relation. So our approach is an {\bf algebraic continuation} in the sense
that we extend the definition of MZVs that preserves the
quasi-shuffle relation.

The QFT renormalization procedure was recently put into a more
mathematical framework through the work of Connes and
Kreimer~\mcite{C-K1,C-K2} and was thus made it possible to be applied to problems in mathematics. For our purpose, the dimensional regularization of
Feynman integrals is replaced by a regularization (or
deformation) of infinite series that has occurred in the study of
Todd classes for toric varieties~\mcite{B-Z}.

Here we illustrate this method by some special cases with the general cases given in \S\,\mref{sec:gmzv} and in~\mcite{G-Z}. First recall the following generating series of Bernoulli numbers that goes back to Euler.
\begin{equation} \frac{\vep}{e^\vep-1}=\sum_{k\geq 0} B_k \frac{\vep^k}{k!}
\mlabel{eq:bern}
\end{equation}
It can be easily rewritten as
\begin{equation}
 \frac{e^\vep}{1-e^\vep}= -\frac{1}{\vep}\frac {-\vep}{e^{-\vep}-1}
    = -\frac{1}{\vep}+ \sum_{k\geq 0}\zeta(-k) \frac{\vep^{k}}{k!}
    \mlabel{eq:zeta}
\end{equation}
since $B_0=1$ and $\zeta(-k)=(-1)^k\frac{B_{k+1}}{k+1}$ for $k\geq
0$.

Now consider
$$Z(s;\vep)=\sum_{n\geq 1} \frac{e^{n\vep}}{n^s}, $$ regarded as a deformation or ``regularization" of the series defining the Riemann zeta function
$\zeta(s)=\sum_{n\geq 1} \frac{1}{n^s}$. The regularized series converges for any integer $s$ when ${\rm Re}(\vep)<0$. In particular,
$$Z(0;\vep)=\frac{e^\vep}{1-e^\vep}$$
and Eq.~(\mref{eq:zeta}) gives
the Laurent series expansion of the regularized sum $Z(0;\vep)=\sum_{n\geq 1}
e^{n\vep}$ at $\vep=0$. For a Laurent series $f(\vep)=\sum_{n\geq
N} a_n\vep^n$ where $N\in \ZZ$, we denote the {\bf power series part} (or the {\bf
finite part}) $\sum_{n\geq 0} a_n\vep^n$ of the Laurent series by
$\fp(f)$, a notation borrowed from~\mcite{M-P1} which, according to
the authors, can be traced back to Hadamard. We tentatively call
$\fp(f)|_{\vep=0}$ the renormalized value of $f(\vep)$ (see
Section~\mref{ss:mzv3} for the general case). Then we have
$$ \fp \big(\sum_{n\geq 1}e^{n\vep}\big)\Big|_{\vep=0} =\zeta(0).$$
So the renormalized value of $Z(0;\vep)=\sum_{n\geq 1} e^{n\vep}$ is $\zeta(0)$. Similarly, to evaluate $\zeta(-k)$ for an integer $k\geq 1$, consider the regularized sum
$$ Z(-k;\vep)=\sum_{n\geq 1}{n^k} {e^{n\vep}}
= \frac{d^k}{d\vep} \big(\frac{e^\vep}{1-e^\vep}\big) $$ which
converges uniformly on any compact subset in ${\rm
Re}(\vep)<0$. So its Laurent series expansion at $\vep=0$ is obtained
by termwise differentiating Eq.~(\mref{eq:zeta}), yielding
\begin{equation}
Z(-k;\vep)=(-1)^{-k-1}(k)!\,\vep^{-k-1}+\sum _{j=0}^{\infty} \zeta
(-k-j)\frac {\vep^j}{j!}. \mlabel{eq:zreg}
\end{equation}
We then have
$$ \fp\big(\sum_{n\geq 1}{n^k} {e^{n\vep}}\big)\Big|_{\vep=0}
    = \zeta (-k).$$

Thus the renormalization method does give the correct
Riemann zeta values at non-positive integers. We next extend this to multiple zeta functions and ``evaluate" $\zeta(0,0)$, for example, by consider the regularized sum
$$ Z(0,0;\vep)=\sum_{n_1>n_2>0} e^{n_1\vep}e^{n_2\vep} =\frac{e^\vep}{1-e^\vep}\frac{e^{2\vep}}{1-e^{2\vep}}.$$
Naively taking the finite part as in the one variable case, we find
$$\fp\big(\sum_{n_1>n_2>0} e^{n_1\vep}e^{n_2\vep}\big)\big|_{\vep=0}
=\fp\big(\frac{1}{2\vep^2}-\frac{3}{2}\zeta(0)\frac{1}{\vep}+
\big(-\frac{5}{2}\zeta(-1)+\zeta(0)^2\big)+o(\vep)\big)\big|_{\vep=0}
=11/24.$$ This does not agree with any of the previously proposed
values of $\zeta(0,0)$ in Eq.~(\mref{eq:jmzv}). Further, this
value does not satisfy the well-known quasi-shuffle
(stuffle) relation:
$$ \zeta(0) \zeta(0)\neq 2\, \zeta(0,0)+\zeta(0)$$
since the left hand side is $1/4$ and the right hand side is $5/12$.
Recall the well-known quasi-shuffle relation
$$ \zeta(s)\zeta(s)=2\, \zeta(s,s)+\zeta(2s)$$
for any integer $s\geq 2$.

To improve this situation and obtain a more suitable definition of
$\zeta(0,0)$, we recall that a key principle in a renormalization  procedure of QFT is that if a divergent Feynman integral contains a component integral that is already
divergent, then the divergency of the component integral should be
removed before removing the divergency of the integral itself.
For our example of $\zeta(0,0)$, the regularized sum
$$\sum_{n_1>n_2>0} e^{n_1\vep}e^{n_2\vep}
=\sum_{n_2\geq 1} e^{n_2\vep} \sum_{n_1\geq n_2+1} e^{n_1\vep}$$
has a component sum $\sum_{n_1\geq n_2+1} e^{n_1\vep}$ that is already divergent when $\vep$ goes to 0.
By the renormalized process adopt to our case (see Eq.~(\mref{eq:dmzv}) and Remark~\mref{re:two0}) we found that the renormalized value should be defined by
\begin{eqnarray*}
&& \fp\Big(\sum_{n_1>n_2>0} e^{n_1\vep}e^{n_2\vep}-\sum_{n_2>0} e^{n_2\vep} \big(
\underbrace{\sum_{n_1>0} e^{n_1 \vep}-{\rm\, \fp}(\sum_{n_1>0} e^{n_1 \vep})}_{\mbox{subdivergence}}\big)\Big)\Big|_{\vep=0}\\
&=& \fp\Big( \frac{1}{2}\frac{1}{\vep^2}-\frac{3}{2}\zeta(0)\frac{1}{\vep} +(-\frac{5}{2}\zeta(-1)+\zeta(0)^2+o(\vep) ) -\big(\frac{1}{\vep^2}-\frac{\zeta(0)}{\vep}-\zeta(-1)+o(\vep) \big) \Big)\Big|_{\vep=0}\\
&=& -\frac{3}{2}\zeta(-1)+\zeta(0)^2=\frac{3}{8}.
\end{eqnarray*}
This value indeed satisfies the
quasi-shuffle relation $ \zeta(0) \zeta(0)= 2\zeta(0,0)+\zeta(0)$.

\medskip

This renormalization process will be systematically carried out for all $\zeta(s_1,\cdots,s_k)$ in \S\,\mref{sec:gmzv} after the general setup in \S\,\mref{sec:setup}.
Remark~\mref{re:two0} shows how the above example
follows from the general case.

\section{Differential Birkhoff decomposition and quasi-shuffle Hopf algebras}
\mlabel{sec:setup}
We formulate a general setup for our later applications to MZVs.

\subsection{Differential Birkhoff decomposition}
We review and extend the algebraic framework of Connes and Kreimer for renormalization of perturbative quantum field theory. For further details of physics applications, see~\mcite{C-K1,C-K2,C-M,E-G-K1,E-G-K2,F-G,Ma}.

Let $\bfk$ be a unitary commutative ring that we usually take to be $\RR$ or a subring of $\CC$.
In the following an algebra means a unitary $\bfk$-algebra unless otherwise specified. A {\bf connected filtered Hopf algebra} is a Hopf
algebra $(H, \Delta)$ with $\bfk$-submodules $H_n,\ n\geq 0,$ of $H$ such
that
\begin{enumerate}
\item $H_n\subseteq H_{n+1}$;
\item $\cup_{n\geq 0} H_n = H$;
\item $H_p H_q\subseteq H_{p+q}$;
\item
$\Delta(H_n) \subseteq \bigoplus_{p+q=n} H_p\otimes H_q.$
\item {\rm (connectedness)} $H_0=\bfk$.
\end{enumerate}
A {\bf Rota--Baxter algebra}~\mcite{Ba,Ro1,Ro4} of weight $\lambda$ is a pair $(R,P)$ where $R$ is an algebra and $P:R\to R$ is a linear operator such that
\begin{equation}
P(x)P(y)=P(xP(y))+P(P(x)y)+\lambda P(xy),
\mlabel{eq:rbe}
\end{equation}
for any $x,\,y\in R$. Often $\theta=-\lambda$ is used, especially in the physics literature.

Now let $H$ be a connected filtered Hopf algebra and let $(R,P)$ be a commutative Rota-Baxter algebra.
Consider the algebra $\calr:=\Hom(H,R)$ of linear maps from $H$ to $R$ where
the product is given by the convolution
$$ f\star g (x):= \sum_{(x)} f(x_{(1)}) g(x_{(2)}).$$
Here we have used the Sweedler's notation
$\Delta(x)=\sum_{(x)} x_{(1)} \barot x_{(2)}$.
Then the operator
$$\calp: \Hom(H,R)\to \Hom(H,R), f\mapsto P\circ f,$$
is a Rota-Baxter operator on $\calr$.

\begin{defn}{\rm
\begin{enumerate}
\item
A {\bf differential algebra} is a pair $(R,d)$ where $R$ is an algebra and $d$ is a {\bf differential operator}, that is, such that  $d(xy)=d(x)y+xd(y)$ for all $x,y\in R$. A differential algebra homomorphism $f: (R_1,d_1)\to (R_2,d_2)$ between two differential algebras $(R_1,d_1)$ and $(R_2,d_2)$ is an algebra homomorphism $f:R_1\to R_2$ such that $f\circ d_1 = d_2 \circ f$.
\item
A {\bf differential Hopf algebra} is a pair $(H,d)$ where $H$ is a
Hopf algebra and $d:H\to H$ is a differential operator such that
\begin{equation}
\Delta(d(x))=
\sum_{(x)} \big( d(x_{(1)})\barot\, x_{(2)} +
    x_{(1)} \barot\, d(x_{(2)})\big).
\mlabel{eq:diffH0}
\end{equation}
\item
A {\bf differential Rota-Baxter algebra} is a triple $(R,P,d)$ where
$(R,P)$ is a Rota-Baxter algebra and $d:R\to R$ is a differential
operator such that $P\circ d=d\circ P$.
\end{enumerate}
}
\mlabel{de:diff}
\end{defn}

\begin{theorem}
Let $H$ be a commutative connected filtered Hopf algebra.
Let $(R,P)$ be a Rota-Baxter algebra of weight $-1$.
Let $\phi: H \to R$ be an algebra homomorphism.
\begin{enumerate}
\item
{\bf (Algebraic Birkhoff Decomposition)} There are algebra homomorphisms $\phi_-: H \to \bfk+P(R)$ and
$\phi_+: H \to \bfk+(\id_R -P)(R)$ such that
$$\phi=\phi_-^{\star\, (-1)}\star \phi_+.$$
Here $\phi_-^{\star\, (-1)}$ is the inverse of $\phi_-$ with respect to the convolution product.
Further,
\begin{equation} \phi_-(x)=-P\big(\phi(x)+\sum_{(x)} \phi_-(x')\phi(x'')\big)
\mlabel{eq:phi-}
\end{equation}
and
\begin{equation}
 \phi_+(x)=(\id-P)\big(\phi(x)+\sum_{(x)} \phi_-(x')\phi(x'')\big).
\mlabel{eq:phi+}
\end{equation}
Here we have used the notation
$\Delta(x)=x\ot 1 + 1\ot x + \sum_{(x)} x'\ot x''.$
\mlabel{it:decom}
\item
If $P^2=P$, then the decomposition in~$($\mref{it:decom}$)$ is unique.
\mlabel{it:uni}
\item
{\bf (Differential Algebraic Birkhoff Decomposition)} If in addition $(H,d)$ is a differential Hopf algebra, $(R,P,\partial)$ is
a commutative differential Rota-Baxter algebra,
and $\phi$ is a differential algebra homomorphism,
then $\phi_-$ and $\phi_+$ are also differential
algebra homomorphisms.
\mlabel{it:diff}
\end{enumerate}
\mlabel{thm:diffBirk}
\end{theorem}
\begin{proof}
For the proof of (\mref{it:decom}), see~\cite[Theorem II.5.1]{Ma}.
For the proof of (\mref{it:uni}), see~\cite[Theorem 3.7]{E-G-K3}. So we just
need to prove (\mref{it:diff}). We only verify that $\phi_-$ is a differential algebra homomorphism, the proof for $\phi_+$ being the same. For this we use induction on $n$
to prove
\begin{equation}
\partial\circ \phi_-(x)=\phi_-\circ d (x)
\mlabel{eq:diff}
\end{equation}
for $x\in H_n, n\geq 0$.

For $x\in H_0$, we have $x\in \bfk$. Since $\phi_-$ is an
algebra homomorphism by (\mref{it:decom}), we have
$\phi_-(x)=x$. So $\partial(\phi_-(x))=0=\phi_-(d(x)).$

Assume Eq.~(\mref{eq:diff}) has been verified for $n\leq k$ and
consider $x\in H_{k+1}$.
Applying $\partial$ to Eq.~(\mref{eq:phi-}), we have
\begin{eqnarray*}
\partial(\phi_-(x))&=& \partial\big(- P(\phi(x)+\sum_{(x)}\phi_-(x')\phi(x''))\big)\\
&=& -P\big( \partial(\phi(x))+\sum_{(x)}(\partial \phi_-(x')\phi(x'')
    +\phi_-(x')\partial(\phi((x'')))\big)\\
&=& -P\big( \phi(d(x))+\sum_{(x)}(\phi_-(d(x'))\phi(x'')
    +\phi_-(x')\phi(d(x''))\big).
\end{eqnarray*}
Here we have used the induction hypothesis in the last step.
On the other hand,
\begin{eqnarray*}
\phi_-(d(x))&=& -P (\phi(d(x))+\sum_{(d(x))}
    \phi_-(d(x)') \phi(d(x)'').
\end{eqnarray*}
Further
$$\Delta(d(x))= d(x)\barot 1 +1\barot d(x)+
    \sum_{(d(x))} d(x)' \barot d(x)''$$
    and
$$d(\Delta(x))
= d\big ( x\barot 1 +1\barot x + \sum_{(x)} x' \barot x'' \big)
= d(x)\barot 1 + 1\barot d(x)
+ \sum_{(x)} \big( d(x') \barot x''+ x' \barot d(x'')\big).
$$
Since $H$ is a differential Hopf algebra with the operator $d$, by Eq.~(\mref{eq:diffH0}), we have
$$ \sum_{(d(x))} d(x)' \barot d(x)''=
\sum_{(x)} \big( d(x') \barot x''+ x' \barot d(x'')\big)$$
and then applying $\phi_-\star \phi$ to it, we get
$$ \sum_{(d(x))} \phi_-(d(x)')  \phi(d(x)'')=
\sum_{(x)} \big( \phi_-(d(x'))  \phi(x'')
    + \phi_-(x')  \phi(d(x''))\big).$$
This completes the induction for Eq.~(\mref{eq:diff}).
\end{proof}

\subsection{Mixable shuffle algebras}
Let $A$ be a unitary or nonunitary $\bfk$-algebra. Consider the tensor product algebra
\begin{equation}
\calh_A:=T(A)=\oplus_{n\geq 0} A^{\mot n}
\mlabel{eq:tenalg}
\end{equation}
with the convention that $A^{\mot 0}=\bfk$.
Here we have used $\mot$ instead of $\ot$ to denote the tensor product of $A$ with itself since $\ot$ will be used for the coproduct introduced below.
Let $\otm$ be the product in $T(A)$, so for $\fraka\in A^{\mot m}$ and $\frakb\in A^{\mot n}$, we have
\begin{equation}
\fraka \otm \frakb =\left \{\begin{array}{ll}
 \fraka\mot\frakb\in A^{\mot m+n},& {\rm\ if\ } m>0,n>0, \\
 \fraka \frakb\in A^{\mot n}, & {\rm if\ } m=0, n>0,\\
\fraka\frakb\in A^{\mot m}, & {\rm if\ }  m>0, n=0,\\
\fraka\frakb\in \bfk, & {\rm if\ } m=n=0.
\end{array} \right .
\mlabel{eq:multm}
\end{equation}
Here the products in the second and third case are scalar product
and in the fourth case is the product in $\bfk$.
Thus, $\otm$ identifies $\bfk\mot A$ and $A\mot \bfk$ with $A$ by the structure maps $\bfk\mot A\to A$ and $A\mot \bfk \to A$ of the $\bfk$-module $A$.

We next equip $\calh_A=T(A)$ with another product $\msh$ so that, together with the deconcatenation coproduct, $\calh_A$ is a connected filtered Hopf algebra. In a special case, this recovers the quasi-shuffle Hopf algebra of Hoffman~\mcite{Ho2}.
The product $\ast$ is defined by the {\bf quasi-shuffle product} in its recursive form, and by the {\bf mixable shuffle product} in its explicit form by one of the authors and Keigher~\mcite{G-K1}.
Given two pure tensors $\fraka\in A^{\mot m}$ and $\frakb\in A^{\mot n}$, to define $\fraka * \frakb$ explicitly, recall that a shuffle of $\fraka$ and $\frakb$
is a tensor list of $a_i$ and $b_j$ without changing the order of the $a_i$s and $b_j$s.
A mixable shuffle is a shuffle in which some (or none) of the pairs $a_i\mot b_j$ are merged
into $a_i b_j$. See~\mcite{G-K1} for the precise definition. Then
$\fraka * \frakb$ is the sum of mixable shuffles of $\fraka$ and $\frakb$. It is shown in~\mcite{E-G1} that the quasi-shuffle product and the mixable shuffle agree with each other.
We note that the mixable shuffle arises from the construction of free commutative Rota--Baxter algebras and has been denoted by $\shplus$ in the literature of Rota--Baxter algebras (previously known as Baxter algebras), such
as~\mcite{A-G-K-O,E-G1,E-G4,E-G5,Gu1,Gu2,Gu3,G-K1,G-K2}. We adapt the notation $\ast$ to be consistent with the current convention on multiple zeta values.

We give more details of the recursive definition of the product $\msh$. We only need to define
$
\fraka \msh \frakb
$
for pure tensors $\fraka\in A^{\mot m}$ and $\frakb\in A^{\mot n}$, $m,n\geq 0$, and then to extend it to a product on $\calh_A$ by biadditivity. If $m=0$, then $\fraka\in \bfk$. Then we define $\fraka\msh \frakb$ to be the scalar product. Similarly if $n=0$. In particular, the identity $\bfone$ of $\bfk$ is also the identity for $\msh$. So we only need to define $\fraka\msh \frakb$ when $m>0$ and $n>0$. Then we have $\fraka=a_1\otm \fraka'$ and $\frakb=b_1\otm \frakb'$ for $a_1,b_1\in A$ and $\fraka'\in A^{\mot (m-1)},\frakb'\in A^{\mot (n-1)}$. With this notation, we define $\fraka\msh \frakb$ by induction on $m+n$. When $m+n=0$, this has been done. Assuming that $\fraka\msh \frakb$ has been define for $0\leq m+n\leq k$ and consider $\fraka\in A^{\mot m},\frakb\in A^{\mot n}$ with $m+n=k+1$. If either $m$ or $n$ is zero, then again we are done. If $m>0$ and $n>0$, then define
\begin{equation}
\fraka \msh \frakb=a_1\otm (\fraka'\msh \frakb)+b_1\otm (\fraka \msh \frakb')+ (a_1b_1) \otm (\fraka'\msh \frakb').
\mlabel{eq:msh}
\end{equation}
By the induction hypothesis, the three terms on the right hand side are defined.

Without resorting to the notation $\otm$, the product $\msh$ is defined by the following recursion.
First define the multiplication by $A^{\mot 0}=\bfk$ to be the scalar product. In particular, $\bfone$ is the identity.
For any $m,n\geq 1$ and
$\fraka:=a_1\mot\cdots \mot a_m\in A^{\mot m}$, $\frakb:=b_1\mot \cdots\mot b_n\in A^{\mot n}$, define
$a \ast b$ by induction on the sum $m+n$. Then $m+n\geq 2$. When $m+n=2$, we have
$a=a_1$ and $b=b_1$. Define
\begin{equation} a\msh b = a_1\mot b_1 + b_1\mot a_1 + a_1b_1.
\label{eq:quasi0}
\end{equation}
Assume that $\fraka\msh \frakb$ has been defined for $m+n\geq k\geq 2$ and consider $\fraka$ and $\frakb$
with $m+n=k+1$. Then $m+n\geq 3$ and so at least one of $m$ and $n$ is greater than 1.
Then we define
\begin{eqnarray}
  \fraka \msh \frakb &=& a_1\mot  b_1\mot  \cdots \mot b_n  + b_1\mot \big(a_1\msh (b_2\mot \cdots\mot b_n)\big) \notag\\
&&  + (a_1b_1)\mot  b_2\mot \cdots\mot  b_n, {\rm\ when\ } m=1, n\geq 2, \label{eq:quasi21} \\
  \fraka \msh \frakb &=& a_1 \mot\big ((a_2\mot \cdots\mot  a_m)\msh b_1 \big) + b_1\mot a_1\mot \cdots\mot a_m \notag \\
&&+ (a_1b_1) \mot  a_2\mot \cdots\mot  a_m,     {\rm\ when\ } m\geq 2, n=1, \label{eq:quasi22} \\
  \fraka \msh \frakb &=& a_1\mot  \big ((a_2\mot \cdots\mot a_m)\msh
(b_1\mot \cdots\mot  b_n)\big ) + b_1\mot  \big ((a_1\mot  \cdots \mot a_m)\msh (b_2 \mot \cdots \mot  b_n)\big) \notag \\
&& \qquad  + (a_1, b_1)  \big ( (a_2\mot \cdots\mot a_m) \msh
     (b_2\mot \cdots\mot  b_n)\big ),
     {\rm\ when\ } m, n\geq 2.\label{eq:quasi23}
\end{eqnarray}
Here the products by $\msh$ on the right hand side of each equation are well-defined
by the induction hypothesis.

To relate to the quasi-shuffle algebra of Hoffman and to the algebra we will use later, we introduce the following concept.

\begin{defn}{\rm
A semigroup $X$ is called a {\bf filtered semigroup} if there is an
increasing sequence $X_n\subseteq X$, $n\geq 1$, of subsets of $X$ such that $X_m  X_n\subseteq X_{m+n}$, $m,n\geq 1$.
}
\end{defn}
Let
$A_X=\bfk\,X$ be the semi-group ring of $X$ with coefficients in $\bfk$. A special case
is when $X$ is {\bf locally finite}, i.e., $X$ has a grading given by the disjoint union
$X=\coprod_{k\geq 1} X^{(k)}$ with each $X^{(k)}$ finite and such that
$X^{(m)}X^{(n)}\subseteq X^{(m+n)}$, $m,n\geq 1$.

The following theorem is a simple generalization of~\cite[Theorem 3.1, 3.2]{Ho2} and \cite[Theorem 2.5]{E-G1} where $X$ is a locally finite set. The same proof works in our generality with the notation $\otm$.

\begin{theorem}
Let $A$ be a commutative nonunitary $\bfk$-algebra.
\begin{enumerate}
\item
The product $\msh$ equips $\calh_A$ with the structure of a commutative unitary filtered $\bfk$-algebra.
\mlabel{it:qalg}
\item
The product $\msh$ coincides with the mixable shuffle product in~\mcite{G-K1}.
\mlabel{it:mixable}
\item
Equip $\calh_A$ with the deconcatenation coproduct
$$\Delta: \calh_A \to \calh_A \barot \calh_A,$$
\begin{eqnarray}
 \Delta( a_1\mot\cdots \mot a_m)&=&1\barot (a_1\mot \cdots \mot a_m)
+ \sum_{i=1}^{m-1} (a_1\mot \cdots\mot a_i)\barot (a_{i+1}\mot \cdots \mot a_m) \notag\\
&&+ (a_1\mot \cdots \mot a_m) \barot 1
\mlabel{eq:coprod}
\end{eqnarray}
and the projection counit
$$\vep: \calh_A \to \bfk$$
onto the direct summand $\bfk=A^{\mot 0}\subseteq \calh_A$.
Then $\calh_A$ is a commutative cocommutative connected filtered Hopf algebra.
\mlabel{it:hopf}
\end{enumerate}
\mlabel{thm:hopf}
\end{theorem}

\begin{proof}
(\mref{it:qalg}) The same proof for \cite[Theorem 2.1]{Ho2} applies. Just replace the length of a word by the tensor degree of a tensor.
Alternatively, one can use item (\mref{it:mixable}) which is proved independent of item~(\mref{it:qalg}) and use the fact that the mixable shuffle product is commutative and associative~\cite[Theorem 3.5]{G-K1}.
\smallskip

\noindent
(\mref{it:mixable}) The proof is the same as for \cite[Theorem 2.5]{E-G1} since the proof there only requires that $A$ is a commutative $\bfk$-algebra.
\smallskip

\noindent
(\mref{it:hopf}) The proof that $\calh_A$ is a bialgebra is the same as for \cite[Theorem 3.1]{Ho2} by considering tensor products in place of word concatenations. Since the result is essential for our later applications, we provide some details.

The coassociativity is clear. So to prove that $\calh_A$ is a bialgebra, we only need to show that $\vep$ and $\Delta$ are algebra homomorphisms. For $\vep$, this is clear. For $\Delta$, we prove
\begin{equation}
\Delta(\fraka)\msh \Delta(\frakb)=\Delta(\fraka\msh \frakb)
\mlabel{eq:dcomp}
\end{equation}
for pure tensors $\fraka\in A^{\mot m}$ and $\frakb\in A^{\mot n}$ by induction on $m+n\geq 0$. If $m+n\leq 1$, then at least one of $\fraka$ and $\frakb$ is a scalar in $\bfk$ and Eq.~(\mref{eq:dcomp}) is clear. Suppose the equation has been proved for $0\leq m+n\leq k$ and consider pure tensors $\fraka$ and $\frakb$ with $m+n=k+1$. Then Eq.~(\mref{eq:dcomp}) is again clear if either one of $m$ or $n$ is zero. So we can assume $m>0$ and $n>0$. Then we have  $\fraka=a_1\otm \fraka'$ and $\frakb=b_1\otm \frakb'$ with $a_1,b_1\in A$ and $\fraka'\in A^{\mot (m-1)},\frakb'\in  A^{\mot (n-1)}$.
Let
 $$\Delta(\fraka')=\sum_{(\fraka')}\fraka'_{(1)}\ot \fraka'_{(2)}, \qquad \Delta(\frakb')=\sum_{(\frakb')}\frakb'_{(1)}\ot \frakb'_{(2)}$$
 by Sweedler's notation.
Then by Eq.~(\mref{eq:coprod}), we have
\begin{equation}
 \Delta(\fraka)=1\ot \fraka + \sum_{(\fraka')}(a_1\otm \fraka'_{(1)} )\ot \fraka'_{(2)}, \qquad
\Delta(\frakb)=1\ot \frakb + \sum_{(\frakb')}(b_1\otm \frakb'_{(1)} )\ot \frakb'_{(2)}.
\mlabel{eq:indD0}
\end{equation}
That is
\begin{equation}
 \Delta(\fraka)=1\ot \fraka + a_1\otm \Delta(\fraka'), \qquad
\Delta(\frakb)=1\ot \frakb + b_1\otm \Delta(\frakb').
\mlabel{eq:indD}
\end{equation}
Here $a_1 \otm \Delta(\fraka')$ means $a_1$ is multiplied with the first tensor factors of $\Delta(\fraka')$.
Thus
\begin{eqnarray*}
\Delta(\fraka)\msh \Delta(\frakb)&=&
\sum_{(\fraka')} \big((a_1\otm \fraka'_{(1)})\msh (b_1\otm \frakb'_{(1)})\big) \ot \sum_{(\frakb')}\big( \fraka'_{(2)}\msh \frakb'_{(2)}\big) \\
&& +\sum_{(\fraka')} (a_1\otm \fraka'_{(1)})\ot (\fraka'_{(2)}\msh \frakb)
+ \sum_{(\frakb')} (b_1\otm \frakb'_{(1)})\ot (\fraka \msh \frakb'_{(2)})
+ 1\ot (\fraka \msh \frakb).
\end{eqnarray*}
Applying Eq.~(\mref{eq:msh}) to the first and fourth terms gives
\begin{align*}
& \sum_{(\fraka')}\big (
a_1\otm (\fraka'_{(1)}\msh (b_1\otm \frakb'_{(1)})) +
b_1\otm ((a_1\otm \fraka'_{(1)})\msh \frakb'_{(1)}) +
(a_1b_1)\otm (\fraka'_{(1)}\msh \frakb'_{(1)})
\big ) \ot \sum_{(\frakb')}\big( \fraka'_{(2)}\msh \frakb'_{(2)}\big) \\
& +\sum_{(\fraka')} (a_1\otm \fraka'_{(1)})\ot (\fraka'_{(2)}\msh \frakb)
+ \sum_{(\frakb')} (b_1\otm \frakb'_{(1)})\ot (\fraka \msh \frakb'_{(2)})\\
&+ \bfone\ot (a_1\otm (\fraka'\msh \frakb)+b_1\otm (\fraka \msh \frakb')+ (a_1b_1) \otm (\fraka'\msh \frakb')).
\end{align*}

On the other hand, by Eq.~(\mref{eq:msh}), Eq.~(\mref{eq:indD}) and the induction hypothesis, we have
\begin{eqnarray*}
\Delta(\fraka\msh \frakb)&=& \Delta\big(a_1\otm (\fraka'\msh \frakb)+b_1\otm (\fraka \msh \frakb')+ (a_1b_1) \otm (\fraka'\msh \frakb)\big)\\
&=& \bfone \ot \big( a_1\otm (\fraka'\msh \frakb)\big)
+ a_1 \otm \Delta(\fraka'\msh \frakb)
+ \bfone \ot \big(b_1\otm (\fraka \msh \frakb')\big) \\
&&+ b_1\otm \Delta(\fraka \msh \frakb')
+ \bfone \ot \big((a_1b_1) \otm (\fraka'\msh \frakb)\big)
+ (a_1b_1) \otm \Delta(\fraka'\msh \frakb)\\
&=& \bfone \ot \big( a_1\otm (\fraka'\msh \frakb)\big)
+ a_1 \otm \big(\Delta(\fraka')\msh \Delta(\frakb)\big)
+\bfone \ot \big(b_1\otm (\fraka \msh \frakb')\big) \\
&&+ b_1\otm \big(\Delta(\fraka) \msh \Delta(\frakb')\big)
+ \bfone \ot \big((a_1b_1) \otm (\fraka'\msh \frakb)\big)
+ (a_1b_1) \otm \big(\Delta(\fraka')\msh \Delta(\frakb')\big)
\end{eqnarray*}
Then by Eq.~(\mref{eq:indD0}), the right hand side is
\begin{align*}
&a_1 \otm \big((\sum_{(\fraka')}\fraka'_{(1)} \ot \fraka'_{(2)})\msh (\bfone \ot \frakb
+ \sum_{(\frakb')} (b_1\otm \frakb'_{(1)}) \ot \frakb'_{(2)}) \big)
 \\
&+ b_1\otm \big( (\bfone\ot \fraka
+\sum_{(\fraka')} (a_1\otm (\fraka'_{(1)}\ot \fraka'_{(2)})) \msh \sum_{(\frakb')} \frakb'_{(1)}\ot \frakb'_{(2)}\big)\\
&+ (a_1b_1) \otm \big(\sum_{(\fraka')}(\fraka'_{(1)} \ot \fraka'_{(2)})\msh
\sum_{(\frakb')}(\frakb'_{(1)}\ot \frakb'_{(2)}) \big)\\
&+ \bfone\ot (a_1\otm (\fraka'\msh \frakb)+b_1\otm (\fraka \msh \frakb')+ (a_1b_1) \otm (\fraka'\msh \frakb')).
\end{align*}
This agrees with $\Delta(\fraka)\msh \Delta(\frakb)$.

Therefore $\calh_A$ is a bialgebra. By the definition of $\msh$ and $\Delta$, $\calh_A$ is connected filtered. Then $\calh_A$ is automatically a Hopf algebra by ~\cite[Proposition 5.3]{F-G}, for example.
\end{proof}

We next give a differential version of quasi-shuffle Hopf algebras.

\begin{theorem}
Let $(A,d)$ be a commutative differential algebra.
Extending $d$ to $\calh_A=\oplus_{k\geq 0} A^{\mot k}$ by defining, for $\fraka:=a_1\mot \cdots\mot a_k\in A^{\mot k}$,
\begin{equation}
 d(\fraka)=\sum_{i=1}^k a_{i,1}\mot\cdots\mot a_{i,k},
 \mlabel{eq:diffH}
 \end{equation}
where
$$ a_{i,j}=\left \{ \begin{array}{ll} a_j, & j\neq i, \\
    d(a_j),& j=i. \end{array} \right . $$
Then $(\calh_A,d)$ is a differential Hopf algebra. \mlabel{thm:diffHopf}
\end{theorem}
\begin{proof}
We have
\begin{eqnarray*}
\Delta\circ d(\fraka)&=&
\sum_{i=1}^k \Delta(a_{i,1}\mot \cdots \mot a_{i,k})\\
&=& \sum_{i=1}^k \sum_{j=0}^k
(a_{i,1}\mot \cdots \mot a_{i,j}) \barot
(a_{i,j+1}\mot \cdots \mot a_{i,k}).
\end{eqnarray*}
Here we have use the convention that if $j=0$, then
$a_{i,1}\mot \cdots \mot a_{i,j}=\bfone\in \bfk$ and
if $j=k$, then $a_{i,j+1}\mot \cdots \mot a_{i,k}=\bfone\in \bfk$.

With the same convention, we also have
{\allowdisplaybreaks
\begin{eqnarray*}
d \circ \Delta(\fraka)&=&
d( \sum_{j=0}^k (a_1\mot \cdots \mot a_j)\barot
    (a_{j+1}\mot \cdots \mot a_k) \\
    &=&
\sum_{j=0}^k \big( d(a_1\mot \cdots \mot a_j) \barot
    (a_{j+1}\mot \cdots \mot a_k) +
    (a_1\mot \cdots \mot a_j) \barot
    d(a_{j+1}\mot \cdots \mot a_k)\big) \\
&=& \sum_{j=0}^k \big( \sum_{i=1}^j (a_{i,1}\mot \cdots \mot a_{i,j}) \barot
    (a_{i,j+1}\mot \cdots \mot a_{i,k}) \\
    && +
    \sum_{i=j+1}^k (a_{i,1}\mot \cdots \mot a_{i,j}) \barot
    (a_{i,j+1}\mot \cdots \mot a_{i,k})\big)\\
&=& \Delta\circ d(\fraka).
\end{eqnarray*}
}
\end{proof}

\section{Renormalization of multiple zeta values and the differential structure}
\mlabel{sec:gmzv}
We now apply the general setup in \S\,\mref{sec:setup} to the study of multiple zeta values. In order to apply the Differential Algebraic Birkhoff Decomposition in Theorem~\mref{thm:diffBirk}, we need to construct a commutative connected differential Hopf algebra $(\calh,d)$, a commutative differential Rota-Baxter algebra $(R,P,\partial)$ of weight $-1$ and an algebra homomorphism $\phi: \calh\to R$. We will provide these in \S\,\mref{ss:mzv3} which is modified with a differential twist from~\mcite{G-Z} for which we refer the reader for some of the details. We then use this decomposition to study the regularized and renormalized MZVs.

\subsection{Differential Algebraic Birkhoff Decomposition for multiple zeta values}
\mlabel{ss:mzv3}

\subsubsection{Directional regularized multiple zeta values}
As mentioned in \S\,\mref{sec:mzvs}, the expression
\begin{equation}
\zeta (\vec{s})= \sum_{n_1>\cdots>n_k>0} \frac{1}{n_1^{s_1} \cdots
    n_k^{s_k}}
\mlabel{eq:formgmzv}
\end{equation}
makes sense for integers $s_1,\cdots,s_k$ only when $s_i>0$ and $s_1>1$. For other integers, we call the expression {\bf formal multiple zeta values} since they are only formal expressions without a real meaning.
In order to make sense of such formal expressions, we
define the {\bf directional regularized multiple zeta values}
\begin{equation}
Z(\wvec{\vec{s}}{\vec{r}};\vep):=\sum_{n_1>\cdots>n_k>0}
\frac{e^{n_1\,r_1\vep} \cdots
    e^{n_k\,r_k\vep}}{n_1^{s_1}\cdots n_k^{s_k}}
\mlabel{eq:reggmzv}
\end{equation}
where $r_i$ are positive real numbers that are introduced so that the space spanned by the directional regularized multiple zeta values is closed under multiplication.

$Z(\wvec{\vec{s}}{\vec{r}};\vep)$ can be defined
recursively as follows. Consider the following set of functions.
For $(s,r)\in \ZZ\times \RR_{>0}$ and $(\vep,x)\in \CC \times \RR$
with ${\rm Re}(\vep)<0$, define
$$ f_{\wvec{s}{r}}(\vep,x):=f(\wvec{s}{r};\vep,x)= \frac{e^{xr\vep}}{x^s}.$$
For vectors $\vec{s}=(s_1,\cdots,s_k)\in \ZZ^k$ and $\vec{r}=(r_1,\cdots,r_k) \in (\RR_{>0})^k$, define
\begin{equation}
Z(\wvec{\vec{s}}{\vec{r}};\vep,x)
=\sum_{n_1>\cdots>n_k>0} \frac{e^{(n_1+x)\,r_1\vep} \cdots
    e^{(n_k+x)\,r_k\vep}}{(n_1+x)^{s_1}\cdots (n_k+x)^{s_k}}.
\end{equation}
Then $Z(\wvec{\vec{s}}{\vec{r}};\vep,x)$ is also given by the recursive definition
$$ Z (\wvec{s}{r};\vep; x)= Q(f_{\wvec{s}{r}}(\vep;x))$$
where $Q$ is the summation operator~\mcite{E-G5,Zud}
\begin{equation}
Q(f)(x)=\sum_{n\geq 1} f(x+n),
\mlabel{eq:sum}
\end{equation}
 and, for
$\vec{s}=(s_1,\cdots,s_k)\in \ZZ^k$ and $\vec{r}=(r_1,\cdots,r_k)
\in (\RR_{>0})^k$,
$$ Z(\wvec{\vec{s}}{\vec{r}};\vep,x)
=Q\big( f_{\wvec{s_k}{r_k}}(\vep,x) Z(\wvec{s_1,\cdots,s_{k-1}}{r_1,\cdots,r_{k-1}};\vep,x)\big), k\geq 2.$$
These are related to the multiple Lerch functions~\mcite{Ca1,E-G5} which are generalizations
of the multiple polylogarithms
$${\rm Li}_{s_1,\cdots,s_k}(z_1,\cdots,z_k)=\sum_{n_1>\cdots n_k>0}
    \frac{z_1^{n_1} \cdots z_k^{n_k}}{n_1^{s_1}\cdots n_k^{s_k}}$$
We then have
$$Z(\wvec{\vec{s}}{\vec{r}};\vep,0) =Z(\wvec{\vec{s}}{\vec{r}};\vep).$$

\subsubsection{The (differential) Hopf algebra for regularized multiple zeta values}

We next construct a Hopf algebra from the regularized MZVs to capture the algebra properties of these values.

We consider the commutative semigroup
\begin{equation}
\frakM= \{f_{{\wvec{s}{r}}}\ |\ (s,r)\in \ZZ \times \RR_{>0}\}
\end{equation}
with the multiplication
$$ f_{\wvec{s}{r}} f_{\wvec{s'}{r'}}=f_{\wvec{s+s'}{r+r'}}.$$
We similarly define the commutative semigroups
\begin{equation}
 \frakM_+=\{f_{\wvec{s}{r}}\ |\ (s,r)\in \ZZ_{>0} \times \RR_{>0}\},\quad
\frakM_-=\{f_{\wvec{s}{r}}\ |\ (s,r)\in \ZZ_{<0} \times \RR_{>0}\}.
\mlabel{eq:mbase}
\end{equation}

For each of these semigroups $\frakN$, let $A_\frakN=\bfk\, \frakN$ be the semigroup algebra. By Theorem~\mref{thm:hopf},
$$\calh_{\frakN}:=\sum_{n\geq 0} (A_\frakN) ^{\mot n}
    =\sum_{n\geq 0} \bfk\, \frakN^n,$$
 with the quasi-shuffle product and deconcatenation coproduct, is a
connected filtered Hopf algebra.
Further, for $\frakN=\frakM$ or $\frakM_-$, the operator
$$ d: A_\frakN \to A_\frakN,\ \wvec{s}{r}=r\wvec{s-1}{r}$$
is a differential operator. Thus by Theorem~\mref{thm:diffHopf}, $\calh_\frakN$ is a differential Hopf algebra.

\subsubsection{The Laurent series for directional regularized multiple zeta values}

Let $\CC\{\{\vep\}\}[\vep^{-1}]$ be the algebra of convergent
Laurent series, regarded as a subalgebra of the algebra of (germs
of) complex valued functions meromorphic in a neighborhood of
$\vep=0$. Since $\ln (-\vep) $ is transcendental over
$\CC\{\{\vep\}\}[\vep]$ by~\cite[Lemma 3.1]{G-Z}, we have
$$\CC\{\{\vep\}\}[\vep^{-1}] [\ln (-\vep)]\cong
\CC\{\{\vep\}\}[\vep^{-1}] [T],$$ the polynomial algebra with the
variable $T=-\ln(-\vep)$ and with
coefficients in $\CC\{\{\vep\}\}[\vep^{-1}]$
which embeds into the ring of Laurent series
$\CC[T][\vep^{-1},\vep]]$
with coefficients in
$\CC[T]$ by the remark after Lemma 3.2 in~\cite{G-Z}.
Thus we obtain an injective algebra homomorphism
\begin{equation}
u: \CC\{\{\vep\}\}[\vep^{-1}] [T]  \to \CC[T][\vep^{-1},\vep]].
\mlabel{eq:formlau}
\end{equation}

Then with the decomposition
$$ \CC[\vep^{-1},\vep]] = \vep^{-1} \CC[\vep^{-1}] \oplus \CC[[\vep]]$$
of $\CC[\vep^{-1},\vep]]$ into a direct sum of subalgebras, we have the vector space direct sum decomposition
$$\CC[T][\vep^{-1},\vep]] =\CC[T][\vep^{-1},\vep]]
= \vep^{-1} \CC[T] [\vep^{-1}] \oplus \CC[T][[\vep]]$$
of the subalgebras $\vep^{-1} \CC[T] [\vep^{-1}]$ and $\CC[T][[\vep]]$. Thus $\CC[T][\vep^{-1},\vep]]$ is a Rota-Baxter algebra with the Rota-Baxter operator $P$ to be the projection to $\vep^{-1} \CC[T] [\vep^{-1}]$:
$$ P\left ( \sum_{n\geq N} \alpha_k(T) \vep^k\right )=\sum_{k\leq -1}\alpha_k(T) \vep^k.$$
This can also be directly verified as with the case of $\CC[\vep^{-1},\vep]]$ (see~\mcite{C-K1,E-G-K1}).

The following facts are easy to verify.
\begin{prop}
\begin{enumerate}
\item
Define the operator
$$\partial:  \CC[T][\vep^{-1},\vep]] \to \CC[T][\vep^{-1},\vep]]$$
by $\partial(\vep^k)=k\vep^{k-1}$ and $\partial(T^n)=nT^{n-1}/\vep$.
Then $\partial$ is a differential operator on $\CC[T][\vep^{-1},\vep]]$.
\item We have
\begin{equation}
\partial (Z(\wvec{\vec{s}}{\vec{r}};\vep,x))
=\sum_{i=1}^k r_i Z(\wvec{\vec{s}-\vec{e}_i}{\vec{r}};\vep,x),
\mlabel{eq:zdiff}
\end{equation}
where $\vec{e}_i\in \ZZ^{k}$ is the $i$-th unit vector.
\item
$(\CC[\vep^{-1},\vep]],P,\partial)$ is a commutative differential Rota-Baxter algebra in the sense of Definition~\mref{de:diff}.
\end{enumerate}
\mlabel{pp:negalg}
\end{prop}

\begin{remark}
We note that the differential operator $\partial$ and the Rota-Baxter operator $P$ do not commute in general. For example,
$\partial (P(T))=\partial (0)=0$, but $P(\partial(T))=P(1/\vep)=1/\vep$.
\end{remark}

By \cite[Theorem 3.3]{G-Z},
$$Z(\wvec{\vec s}{\vec r};\vep)=\sum_{n_1>\cdots
>n_k>0} \frac {e^{n_1r_1\vep} e^{n_2r_2\vep} \cdots e^{n_kr_k\vep}}{n_1^{s_1}n_2^{s_2}\cdots
n_k^{s_k}}$$
has a Laurent series expansion in
$\CC [\ln(-\vep)]\{\{\vep\}\}[\vep^{-1}]$
for $\vec s\in \Bbb {Z} ^k$, $\vec r \in \Bbb
{Z}_{>0}^k$ and a Laurent series expansion in
$\CC\{\{\vep\}\}[\vep^{-1}]$
for $\vec s\in \Bbb {Z}_{\le 0}^k$,
$\vec r \in \Bbb {Z}_{>0}^k$.
Combining this with Eq.~(\mref{eq:formlau}), we obtained an algebra homomorphism
$$
\uni{Z}: \calh_\frakM \to \CC[T][\vep^{-1},\vep]],
\quad \uni{Z}(f_{\wvec{\vec{s}}{\vec{r}}})=u(Z(\wvec{\vec{s}}{\vec{r}};\vep))
$$
 which restricts to a differential algebra homomorphism
$$
\uni{Z}: \calh_{\frakM^-} \to \CC[\vep^{-1},\vep]].
$$
Then by the Differential Algebraic Birkhoff Decomposition in Theorem~\mref{thm:diffBirk},
we have

\begin{coro}
There is the unique decomposition
$$ \uni{Z}=\uni{Z}_-^{-1} \star \uni{Z}_+$$
where the map $\uni{Z}_+: \calh_\frakM\to \CC[T][[\vep]]$ is an algebra homomorphism which restricts to a differential algebra homomorphism
$\uni{Z}_+: \calh_{\frakM^-}\to \CC[[\vep]]$.
\mlabel{co:renormz}
\end{coro}

By Theorem~3.8 in~\mcite{G-Z},
\begin{equation}
\uni{Z}_+\lp \wvec{s_1,s_2}{r_1,r_2};\vep\rp
=(\id-P) \Big( \uni{Z}\lp \wvec{s_1,s_2}{r_1,r_2};\vep\rp
- P\lp \uni{Z}\lp \wvec{s_1}{r_1};\vep\rp \rp \uni{Z}\lp \wvec{s_2}{r_2};\vep\rp \Big) .
\mlabel{eq:zplus2}
\end{equation}

\begin{defn}
For $\vec{s}=(s_1,\cdots,s_k)\in \ZZ^k$ and $\vec{r}=(r_1,\cdots,r_k)\in \RR_{>0}^k$, define
the {\bf renormalized directional multiple zeta value (MZV)}
\begin{equation} \zeta\lp \wvec{\vec{s}}{\vec{r}}\rp = \lim_{\vep\to 0} \uni{Z}_+\lp \wvec{\vec{s}}{\vec{r}};\vep\rp .
\mlabel{eq:dmzv}
\end{equation}
\end{defn}
The definition makes sense because of Corollary~\mref{co:renormz}.
Furthermore, since $\uni{Z}_+$ is an algebra homomorphism, we have
the quasi-shuffle relation
$$ \zeta\lp \wvec{\vec{s}}{\vec{r}}\rp \zeta\lp \wvec{\vec{s}'}{\vec{r}'}\rp
= \zeta\lp \wvec{\vec{s}}{\vec{r}} \msh \wvec{\vec{s}'}{\vec{r}'} \rp $$
meaning if
$$ \wvec{\vec{s}}{\vec{r}} \msh \wvec{\vec{s}\,'}{\vec{t}\,'} =\sum_{\wvec{\vec{s}\,''}{\vec{t}\,''}} \wvec{\vec{s}\,''}{\vec{t}\,''}$$
is the quasi-shuffle product in $\calh_\frakM$ defined in Eq.~(\mref{eq:msh}) -- (\mref{eq:quasi23}), then
$$ \zeta\lp \wvec{\vec{s}}{\vec{r}}\rp \zeta\lp \wvec{\vec{s}'}{\vec{r}'}\rp
= \sum_{\wvec{\vec{s}\,''}{\vec{t}\,''}}
    \zeta\big(\wvec{\vec{s}\,''}{\vec{t}\,''}\big ).$$

\begin{defn}
For $\vec{s}\in \ZZ_{> 0}^k\cup \ZZ_{\leq 0}^k$,  define
\begin{equation} \gzeta\lp \vec{s}\rp
= \lim_{\delta \to 0^+} \zeta\lp \wvec{\vec{s}}{|\vec{s}|+\delta}\rp ,
\mlabel{eq:gmzv}
\end{equation}
where, for $\vec s=(s_1,\cdots,s_k)$ and $\delta\in \RR_{>0}$, we denote $|\vec{s}|=(|s_1|,\cdots,|s_k|)$ and $|\vec s|+\delta=(|s_1|+\delta,\cdots,|s_k|+\delta).$
These $\gzeta(\vec{s})$ are called the {\bf renormalized MZVs} of the multiple zeta function $\zeta(u_1,\cdots,u_k)$ at $\vec{s}$.
\mlabel{de:rmzv}
\end{defn}

We show in \cite[Theorem4.2]{G-Z} that our renormalized MZVs are
well-defined and are compatible with the known MZVs defined by
either convergence, analytic continuation~\mcite{Zh2} or the
Ihara-Kaneko-Zagier regularization~\mcite{I-K-Z}. It is also proved that the renormalized MZVs
satisfy the quasi-shuffle relation.

We next show that our examples in Section~\mref{sec:mzvs} follows from the above general set up. Let $\Sigma_k$ denote the symmetric group on $k$ letters. For $\sigma\in \Sigma_k$ and $\vec{r}=(r_1,\cdots,r_k)$, denote
$\sigma(\vec{r})=(r_{\sigma(1)},\cdots,r_{\sigma(k)})$ and $f(\vec r)^{(\Sigma_k)}=\sum _{\sigma \in \Sigma_k} f(\sigma (\vec r))$.

\begin {prop} {\rm (\cite[Proposition~4.9]{G-Z})} \mlabel{pp:gzero}
For any $k\geq 1$,
$\zeta\lp\wvec{\vec 0_k}{\vec r}\rp^{(\Sigma_k)}$ is independent of the choice of  $\vec{r}\in \RR_{>0}^k$ and
$$\gzeta (\vec 0_k)
=\frac 1{k!}\zeta\lp\wvec{\vec 0_k}{\vec r}\rp^{(\Sigma_k)}.$$
\end{prop}
\begin{remark} {\rm
Taking $k=2$ and $\vec{r}=(1,1)$ in the above proposition, we have
$\gzeta(0,0)=\zeta\big ( \wvec{0,0}{1,1}\big)$. Then by
Eq.~(\mref{eq:zplus2}) and Eq.~(\mref{eq:dmzv}),
we obtain the renormalized value for $\zeta(0,0)$
discussed in Section~\mref{sec:mzvs}.
}
\mlabel{re:two0}
\end{remark}

\subsection{Differential structures on renormalized multiple zeta values}
\mlabel{sec:diff}

We now consider some differential properties of the renormalized MZVs.
By Eq.~(\mref {eq:zreg}),
$$
\uni{Z}(\wvec{0}{r};\vep)=-\frac 1{r\vep}+\sum _{i=0}^{\infty} \zeta
(-i)\frac {(r\vep )^i}{i!}.
$$
So
\begin{equation}
\uni{Z}_+(\wvec{0}{r};\vep)=\sum _{i=0}^{\infty} \zeta
(-i)\frac {(r\vep )^i}{i!}
\mlabel{eq:gen0}
\end{equation}
is the generating function for $\zeta (s)$, $s\in \ZZ _{\le
0}$. We next generalize this to multiple variables.

\begin{theorem}
\begin{equation}
\uni{Z}_+\big(\wvec{\vec{0}_k}{\vec{r}};\vep\big)
= \sum_{n\geq 0} \sum_{i_1+\cdots+i_k=n}
\zeta\big(\wvec{-(i_1,\cdots,i_k)}{\vec{r}}\big)
\frac{(r_1\vep)^{i_1}}{i_1!}\cdots \frac{(r_k\vep)^{i_k}}{i_k!},
\mlabel{eq:geng}
\end{equation}
where the sum runs over ordered partitions of $n$.
\mlabel{eq:neggen}
\end{theorem}

\begin{proof}
Let
$$\uni{Z}_+(\wvec{\vec{s}}{\vec{r}};\vep)=\sum _{n\ge 0}a_n \frac {\vep
^n}{n!}.$$
We find a
formula for $a_n$. Of course
$a_n$ is the constant term of
$$\uni{Z}_+ ^{(n)}(\wvec{\vec{s}}{\vec{r}};\vep)
: = \frac{d^n}{d \vep^n}\uni{Z}_+ (\wvec{\vec{s}}{\vec{r}};\vep).$$
But since $\uni{Z}_+$ is a differential algebra homomorphism by Corollary~\mref{co:renormz}, we have
$$
\frac{d}{d\vep}\uni{Z}_+(\wvec{\vec{s}}{\vec{r}};\vep)=\sum _{i=1}^k
r_i \uni{Z}_+(\wvec{\vec{s}-\vec{1}_i}
{\vec{r}};\vep)
$$
where $\vec{e}_i$ is the $i$-th unit vector of length $k$.
By an inductive argument, we have in general
$$\uni{Z}_+ ^{(n)}(\wvec{\vec{s}}{\vec{r}};\vep)=\sum _{i_1+\cdots+i_k=n} \binc
{n}{i_1,\cdots,i_k}r_1^{i_1}\cdots r_k^{i_k}\uni{Z}_+(\wvec{\vec{s}-(i_1,\cdots,i_k)}
{\vec{r}};\vep)
$$
where the sum runs over ordered partitions of $n$.
Thus we have
\begin{eqnarray*}
a_n &=& \uni{Z}^{(n)}_+ \big(\wvec{\vec{s}}{\vec{r}}; 0 \big)\\
&& = \sum _{i_1+\cdots+i_k=n} \binc
{n}{i_1,\cdots,i_k}r_1^{i_1}\cdots r_k^{i_k}\uni{Z}_+(\wvec{\vec{s}-(i_1,\cdots,i_k)}
{\vec{r}};0) \\
&& = \sum _{i_1+\cdots+i_k=n} \binc
{n}{i_1,\cdots,i_k}r_1^{i_1}\cdots r_k^{i_k}\zeta\big(\wvec{\vec{s}-(i_1,\cdots,i_k)}
{\vec{r}}\big)
\end{eqnarray*}
Therefore we have the generating function
\begin{eqnarray*}
\uni{Z}_+\big(\wvec{\vec{s}}{\vec{r}}\big)
&=&
\sum_{n\geq 0} \sum _{i_1+\cdots+i_k=n} \binc
{n}{i_1,\cdots,i_k}r_1^{i_1}\cdots r_k^{i_k}\zeta\big(\wvec{\vec{s}-(i_1,\cdots,i_k)}{\vec{r}}\big) \frac{\vep^n}{n!}\\
&& = \sum_{n\geq 0} \sum_{i_1+\cdots+i_k=n}
\zeta\big(\wvec{\vec{s}-(i_1,\cdots,i_k)}{\vec{r}}\big)
\frac{r_1^{i_1}}{i_1!}\cdots \frac{r_k^{i_k}}{i_k!} \vep^n\\
&&= \sum_{n\geq 0} \sum_{i_1+\cdots+i_k=n}
\zeta\big(\wvec{\vec{s}-(i_1,\cdots,i_k)}{\vec{r}}\big)
\frac{(r_1\vep)^{i_1}}{i_1!}\cdots \frac{(r_k\vep)^{i_k}}{i_k!}
\end{eqnarray*}
This proves Eq.~(\mref{eq:geng}).
\end{proof}

We also consider the regularized series before the renormalization.
In the one variable case, we have
$$
\uni{Z}(\wvec{0}{r};\vep)=-\frac 1{r\vep}+\sum _{i=0}^{\infty} \zeta
(-i)\frac {(r\vep )^i}{i!}
$$
For the two variable case, we have
$$\uni{Z}(\wvec{0,0}{r_1,r_2};\vep)=P_0+\sum _{n\ge 0}a_n \frac {\vep
^n}{n!}$$ where $P_0$ is the negative power part. Let us find a
formula for $a_n$. First by $n$-th derivative, we have
$a_n=$constant term of $\uni{Z} ^{(n)}(\wvec{s_1,s_2}{r_1,r_2};\vep)$.
But
$$\uni{Z} ^{(n)}(\wvec{0,0}{r_1,r_2};\vep)=\sum _{i=0}^n \binc
{n}{i}r_1^ir_2^{n-i}\uni{Z}(\wvec{-i,i-n}{r_1,r_2};\vep),
$$
and, by \cite[Proposition~4.8]{G-Z}, the constant term of $\uni{Z}(\wvec{-i,i-n}{r_1,r_2};\vep)$ is
$$\zeta\lp \wvec{-i,i-n}{r_1,r_2}\rp+(-1)^{i-1}\frac 1{i+1}(\frac
{r_2}{r_1})^{i+1}\zeta (-n-1).$$
So
\begin{eqnarray*}
a_n&=&\sum _{i=0}^n \binc
{n}{i}r_1^ir_2^{n-i}\zeta\lp \wvec{-i,i-n}{r_1,r_2}\rp+\big(\sum
_{i=0}^n \binc {n}{i}(-1)^{i-1}\frac 1{i+1}\big)\frac
{r_2^{n+1}}{r_1}\zeta (-n-1)
\\
&=& \sum _{i=0}^n \binc
{n}{i}r_1^ir_2^{n-i}\zeta\lp \wvec{-i,i-n}{r_1,r_2}\rp-\frac
{r_2^{n+1}}{r_1}\frac {\zeta (-n-1)}{n+1}
\end{eqnarray*}


\end{document}